\documentclass[preprint,10pt]{elsarticle}
\usepackage{amsfonts}
 \usepackage{graphics}
\usepackage{graphicx}
 \usepackage{epsfig}

\usepackage{amssymb}
\usepackage{amsthm}

\newtheorem{thm}{Theorem}

\newdefinition{rmk}{Remark}
\newproof{pf}{Proof}
\newdefinition{example}{Example}
\newdefinition{definition}{Definition}
\newdefinition{proposition}{Proposition}
\newdefinition{corollary}{Corollary}
\journal{Arkiv}

\begin{document}

\begin{frontmatter}

\title{Note on regions containing eigenvalues of a matrix}

\author[rvt1]{Suhua Li}
\author[rvt]{Qingbing Liu}
\author[rvt1]{Chaoqian Li\corref{cor1}}
\ead{lichaoqian@ynu.edu.cn}

\cortext[cor1]{Corresponding author.}
\address[rvt1]{School of Mathematics and Statistics, Yunnan
University, Kunming, P. R. China  650091}

\address[rvt]{Department of
Mathematics, Zhejiang Wanli University, Ningbo, P.R. China }

\begin{abstract}
By excluding some regions, in which each eigenvalue of a matrix is not contained, from the $\alpha\beta$-type eigenvalue inclusion region  provided by Huang et al.(Electronic Journal of Linear Algebra, 15 (2006) 215-224),
a new eigenvalue inclusion region is given. And it is proved that the new region is contained in the $\alpha\beta$-type eigenvalue inclusion region.
\end{abstract}

\end{frontmatter}

% \linenumbers


%%%%%%%%%%%%%%%%%%%%%%%%%%%%%%%%%%%%%%%%%%%%%%%%%%%%%%%%%%
\section{Introduction}Let $A=[a_{ij}]\in \mathbb{C}^{n\times n}$be an $n\times n$ complex matrix with $n\geq 2$ , and  $\alpha$ and $\beta$ nonempty index sets satisfying
\[\alpha  \bigcup \beta= N ~and ~ \alpha  \bigcap\beta  = \emptyset, \]
throughout the paper. Define partial absolute deleted row sums and column sums as follows:
\[r_i^\alpha (A):= \sum\limits_{j\in \alpha, \atop j\neq i} |a_{ij}|, ~c_i^\alpha (A):= \sum\limits_{j\in \alpha, \atop j\neq i} |a_{ji}|,\]
\[r_i^\beta (A):= \sum\limits_{j\in \beta, \atop j\neq i} |a_{ij}|, ~c_i^\beta (A):= \sum\limits_{j\in \beta, \atop j\neq i} |a_{ji}|.\]
If $\alpha =\{i_0\}$, then we assume, by convention, that $r_{i_0}^\alpha (A)=0$. Similarly
 $r_{i_0}^\beta (A)=0$ if $\beta=\{ i_0\}$. Clearly,
\[ r_i(A)= \sum\limits_{j\neq i} |a_{ij}|= r_i^\alpha (A)+ r_i^\beta (A),\]
 and
\[ c_i(A)= \sum\limits_{j\neq i} |a_{ji}|=c_i^\alpha (A)+ c_i^\beta (A).\]

In \cite{Hu}, Huang et al. provided a so-called $\alpha\beta$-type eigenvalue inclusion region by using the partial absolute deleted row sums $r_i^\alpha (A)$ and $r_i^\beta (A)$ as follows.

\begin{thm}\label{huang} \cite[Theorem 2.1]{Hu}
Let $A=[a_{ij}]\in \mathbb{C}^{n\times n}$ and $\lambda$ be an eigenvalue of $A$. Then
 \[\lambda \in  G^{(\alpha \beta)} (A) := \left(\bigcup\limits_{i\in \alpha} G_i^{(\alpha)}(A)\right) \bigcup \left(\bigcup\limits_{j\in \beta} G_j^{(\beta)}(A)\right)
  \bigcup \left( \bigcup \limits_{i\in \alpha, \atop j \in \beta} G^{(\alpha\beta)}_{ij} (A)\right),  \]
where \[G_i^{(\alpha)}(A):=\bigcup \limits_{i\in \alpha}\left\{z\in \mathbb{C}:|a_{ii}-z| \leq   r_i^\alpha(A) \right \},\]
\[ G_j^{(\beta)}(A) :=\bigcup \limits_{j\in \overline{S}}\left\{z\in \mathbb{C}:|a_{jj}-z| \leq   r_j^{\beta}(A) \right \},\]
and
\begin{eqnarray*} G^{(\alpha\beta)}_{ij} (A)&:=& \left\{z\in \mathbb{C}: z\notin  G_i^{(\alpha)}(A) \bigcup G_j^{(\beta)}A), \right.\\
&&\left.\left(|a_{ii}-z|-  r_i^\alpha(A)\right) \left(|a_{jj}-z|-  r_j^{\beta}(A) \right)\leq   r_i^{\beta}(A)r_j^\alpha(A) \right \}.\end{eqnarray*}
\end{thm}

The region  $G^{(\alpha \beta)} (A)$ in Theorem \ref{huang} is equivalent to
\[\left(\bigcap\limits_{i\in \alpha} G_i^{(\alpha)}(A)\right) \bigcup \left( \bigcup \limits_{i\in \alpha, \atop j \in \beta} V^{(\alpha\beta)}_{ij} (A)\right) \] provided by Farid  in \cite{Fa},
where \[V^{(\alpha\beta)}_{ij} (A):=\left\{z\in \mathbb{C}: \left(|a_{ii}-z|-  r_i^\alpha(A)\right) \left(|a_{jj}-z|-  r_j^{\beta}(A) \right)\leq   r_i^{\beta}(A)r_j^\alpha(A) \right \}, \] for details, see Remark 2.4, \cite{Fa}. It seems that each $G_j^{(\beta)}(A)$, $ j\in \beta$ dose not work, which is illustrated by the matrix
\[A= \left[
\begin{array}{cccc}
  1   &-0.01    &0     &7 \\
  0.1     &0.9   &-1.1   &0.3  \\
  -0.11  &0.2   &0.9    &0.2\\
  2  & 0.4   &-0.1   &1.1
\end{array} \right].\] All eigenvalues of $A$ are
\[  4.8161,  -2.6994,   0.8917 + 0.4921\textbf{i}, 0.8917 - 0.4921\textbf{i},\]
where $\textbf{i}^2=-1$. Take $\alpha=\{1,3\}$ and $\beta=\{2,4\}$, then
\[G_2^{(\beta)}(A)=\{ z\in \mathbb{C}: |z-0.9|\leq 0.3 \}, \]
and
\[G_4^{(\beta)}(A)=\{ z\in \mathbb{C}: |z-1.1|\leq 0.4 \}. \]
It is easy to see that all eigenvalues are not contained in $G_2^{(\beta)}(A) \bigcup G_4^{(\beta)}(A) $, see Figure 1.
%\begin{figure}[bft]
%\centering
%\includegraphics[width=4in]{fig1.jpg}
%\caption{All eigenvalues of $A$ and $G_2^{(\beta)}(A) \bigcup G_4^{(\beta)}(A)$ }
%\end{figure}
Hence  $G_2^{(\beta)}(A) \bigcup G_4^{(\beta)}(A) $ can be excluded from $G^{(\alpha \beta)} (A)$ for the matrix $A$.
On the other hand, by taking $\beta=\{2,3\}$, Figure 2 shows that
\[\{ 0.8917 + 0.4921\textbf{i}, 0.8917 - 0.4921\textbf{i} \} \subseteq \left(G_2^{(\beta)}(A) \bigcup G_3^{(\beta)}(A)\right)\]
where $G_2^{(\beta)}(A)=\{ z\in \mathbb{C}: |z-0.9|\leq 1.1 \}$ and $G_3^{(\beta)}(A)=\{ z\in \mathbb{C}: |z-0.9|\leq 0.2 \}.$
This means that each  $G_{j}^{(\beta)}(A)$ for $j=2,3$ cannot be excluded from  $G^{(\alpha \beta)} (A)$ for the this case.
This motivates us to find another regions in which  any eigenvalue of a matrix $A$ is not contained to exclude them.
In this paper, a new inclusion region by excluding some such regions from
 $G^{(\alpha \beta)} (A)$ is given. This provides a sufficient condition for the non-singularity of a matrix.  Numerical examples are also given to verify the corresponding results.

%\begin{figure}[bft]
%\centering
%\includegraphics[width=4in]{fig2.jpg}
%\caption{All eigenvalues of $A$ and $G_2^{(\beta)}(A) \bigcup G_3^{(\beta)}(A)$ }
%\end{figure}
%%%%%%%%%%%%%%%%%%%%%%%%%%%%%%%%%%%%%%%%%%%%%%%%%%

\section{Main results}
In this section, some regions in which  any eigenvalue of a matrix  is not contained, are excluded from  $G^{(\alpha \beta)} (A)$.

\begin{thm}\label{Mian_thm}
Let $A=[a_{ij}]\in \mathbb{C}^{n\times n}$ and $\lambda$ be an eigenvalue of $A$. Then
\begin{eqnarray}\label{eq1}
\lambda \in  E^{(\alpha \beta)} (A)= \left(\bigcup\limits_{i\in \alpha} G_i^{(\alpha)}(A)\right)& \bigcup \left(\bigcup\limits_{j\in \beta} G_j^{(\beta)}(A)\right)\bigcup \left( \bigcup \limits_{i\in \alpha, \atop j \in \beta} \left(G^{(\alpha\beta)}_{ij} (A)\backslash  \widetilde{E}^{(\alpha\beta)}_{ij} (A) \right) \right)\nonumber\\
&\bigcup \left( \bigcup \limits_{i\in \alpha, \atop j \in \beta} \left(G^{(\alpha\beta)}_{ij} (A)\backslash  \widehat{E}^{(\alpha\beta)}_{ij} (A) \right) \right)
\end{eqnarray}
where $G_i^{(\alpha)}(A)$, $G_j^{(\beta)}(A)$ and  $G^{(\alpha\beta)}_{ij} (A)$ are defined in Theorem \ref{huang},
\[\widetilde{E}^{(\alpha\beta)}_{ij} (A) = \left\{z\in \mathbb{C}: (|z-a_{ii}|+r_i^j(A)) \left(|z-a_{jj}|+r_j^\beta(A)\right) < |a_{ij}| \left(2|a_{ji}|-r_j^\alpha(A)\right) \right \},\]
and
\[\widehat{E}^{(\alpha\beta)}_{ij} (A) = \left\{z\in \mathbb{C}:(|z-a_{jj}|+r_j^i(A)) \left(|z-a_{ii}|+r_i^\alpha(A)\right) < |a_{ji}| \left(2|a_{ij}|-r_i^\beta(A)\right)  \right \}.\]
Furthermore, $ E^{(\alpha \beta)} (A)  \subseteq  G^{(\alpha \beta)} (A) $.
\end{thm}

\begin{proof}
Since
\[\left(G^{(\alpha\beta)}_{ij} (A)\backslash  \widetilde{E}^{(\alpha\beta)}_{ij} (A) \right) \subseteq   G^{(\alpha\beta)}_{ij} (A)\]
and
\[ \left(G^{(\alpha\beta)}_{ij} (A)\backslash  \widehat{E}^{(\alpha\beta)}_{ij} (A) \right)  \subseteq   G^{(\alpha\beta)}_{ij} (A), \]
hold for each $i\in \alpha$ and $j\in \beta$,  it implies that
\[  E^{(\alpha \beta)} (A)  \subseteq  G^{(\alpha \beta)} (A).\] Hence, we only need to prove
\[ \lambda \in  E^{(\alpha \beta)} (A).\]

Suppose that $x=[x_1,x_2,\ldots, x_n]^T $ is an eigenvector of $A $ corresponding to the eigenvalue $\lambda$, that is,
\begin{equation}\label{eq2} Ax=\lambda x.\end{equation}
Let $|x_p|=\max\limits_{i\in \alpha} \{|x_i|\}$ and $|x_q|=\max\limits_{j\in \beta} \{|x_j|\}$. Obviously, at least one of $|x_p|$ and  $|x_q|$ is nonzero. Note that if
\[|\lambda-a_{pp}| \leq   r_p^\alpha(A), ~or~ |\lambda-a_{qq}| \leq   r_q^{\beta}(A),\]
then $\lambda \in \left(\bigcup\limits_{i\in \alpha} G_i^{(\alpha)}(A)\right) \bigcup \left(\bigcup\limits_{j\in \beta} G_j^{(\beta)}(A)\right) \subseteq   E^{(\alpha \beta)} (A)$.
Hence, next we only prove that (\ref{eq1}) holds for the case that
\begin{equation}\label{eq3} |a_{pp}-\lambda| >  r_p^\alpha(A), ~and~ |a_{qq}-\lambda| >  r_q^{\beta}(A).\end{equation}

(I) Suppose that $ |x_p|\geq |x_q|$. Similarly to the proof of Theorem 2.1 in \cite{Hu}, we can get that
\begin{equation}\label{eq4} \left(|\lambda-a_{pp}|-r_p^\alpha(A)\right) \left(|\lambda-a_{qq}|-r_q^\beta(A)\right)  \leq r_p^\beta(A) r_q^\alpha(A).\end{equation}
Furthermore, consider the $p$-th equation of (\ref{eq2}) and rewrite it into
\begin{equation}\label{eq5}
(\lambda-a_{pp})x_p-\sum\limits_{j\in N, \atop j\neq p, q} a_{pj}x_j = a_{pq}x_q,
\end{equation}
Taking absolute values on both sides of (\ref{eq5}) and using the triangle inequality give
\begin{eqnarray}\label{eq6}
|a_{pq}||x_q| &\leq& |\lambda-a_{pp}||x_p|+ \sum\limits_{j\in N, \atop j\neq p, q} |a_{pj}||x_j|\nonumber\\
 &\leq&|\lambda-a_{pp}||x_p|+ \sum\limits_{j\in N, \atop j\neq p, q} |a_{pj}||x_p|\nonumber\\
 &=& (|\lambda-a_{pp}|+r_p^q(A))|x_p|.\end{eqnarray}
On the other hand, consider the $q$-th equation of (\ref{eq2}) and rewrite it into
\begin{eqnarray}\label{eq7}
(\lambda-a_{qq})x_q-\sum\limits_{j\in \alpha, \atop j\neq p} a_{qj}x_j -\sum\limits_{j\in \beta, \atop j\neq q} a_{qj}x_j = a_{qp}x_p,
\end{eqnarray}
Taking absolute values on both sides of (\ref{eq7}) and using the triangle inequality give
\begin{eqnarray*}
|a_{qp}||x_p| &\leq& |\lambda-a_{qq}||x_q|+ \sum\limits_{j\in \alpha, \atop j\neq p} |a_{qj}||x_j| +\sum\limits_{j\in \beta,\atop j\neq q} |a_{qj}||x_j|\nonumber\\
 &\leq&|\lambda-a_{qq}||x_q|+ \sum\limits_{j\in \alpha, \atop j\neq p} |a_{qj}||x_p| +\sum\limits_{j\in \beta,\atop j\neq q} |a_{qj}||x_q|\nonumber\\
 &=& \left(|\lambda-a_{qq}|+r_q^\beta(A)\right)|x_q|+ \sum\limits_{j\in \alpha, \atop j\neq p} |a_{qj}||x_p|.\end{eqnarray*}
Hence
\begin{equation}\label{eq8} \left(2|a_{qp}|-r_q^\alpha(A)\right)|x_p| \leq  \left(|\lambda-a_{qq}|+r_q^\beta(A)\right)|x_q|.\end{equation}
If $ |x_q|> 0$, then multiplying (\ref{eq6}) with (\ref{eq8}) gives
\begin{equation}\label{eq9}
|a_{pq}| \left(2|a_{qp}|-r_q^\alpha(A)\right) \leq (|\lambda-a_{pp}|+r_p^q(A)) \left(|\lambda-a_{qq}|+r_q^\beta(A)\right).
\end{equation}
If  $ |x_q|=0$, then by (\ref{eq8}) we have
\[ 2|a_{qp}|-r_q^\alpha(A) \leq 0,\]
and hence (\ref{eq9}) also holds. From (\ref{eq4}) and (\ref{eq9}), it follows that
\[\lambda \in G^{(\alpha\beta)}_{ij} (A) \bigcap \overline{\widetilde{E}^{(\alpha\beta)}_{ij} (A)}= G^{(\alpha\beta)}_{ij} (A) \backslash \widetilde{E}^{(\alpha\beta)}_{ij} (A) \subseteq   E^{(\alpha \beta)} (A),\]
where $\overline{\widetilde{E}^{(\alpha\beta)}_{ij} (A)}$ is the complementary of the set $ \widetilde{E}^{(\alpha\beta)}_{ij} (A)$.

(II) Suppose that $ |x_q|\geq |x_p|$. Similarly to (I), we  can easily obtain that (\ref{eq4}) holds, and
\begin{equation}\label{eqn2.6}
|a_{qp}| \left(2|a_{pq}|-r_p^\beta(A)\right) \leq (|\lambda-a_{qq}|+r_q^p(A)) \left(|\lambda-a_{pp}|+r_p^\alpha(A)\right).
\end{equation}
Hence
\[\lambda \in G^{(\alpha\beta)}_{ij} (A) \bigcap \overline{\widehat{E}^{(\alpha\beta)}_{ij} (A)}= G^{(\alpha\beta)}_{ij} (A) \backslash \widehat{E}^{(\alpha\beta)}_{ij} (A) \subseteq  E^{(\alpha \beta)} (A).\]
From (I) and (II) the conclusion follows.\end{proof}

\begin{rmk} \label{rmk1} Theorem \ref{Mian_thm} shows that for any eigenvalue $\lambda$ of a matrix $A, $
\[ \lambda \notin  \widetilde{E}^{(\alpha\beta)}_{ij} (A), ~or~\lambda \notin  \widehat{E}^{(\alpha\beta)}_{ij} (A) ~for ~j\neq i. \]
And hence $\widetilde{E}^{(\alpha\beta)}_{ij} (A)$ (or $\widehat{E}^{(\alpha\beta)}_{ij} (A)$) can be excluded from the set $G^{(\alpha \beta)} (A)$.
Consider again the matrix $A$ in the introduction section, and take $\alpha=\{1,2\}$ and $\beta=\{3,4\}$. From Figure 3, it is not difficult to see that  all eigenvalues
\[\{ 4.8161,  -2.6994,  0.8917 + 0.4921\textbf{i}, 0.8917 - 0.4921\textbf{i}\} \subseteq E^{(\alpha \beta)} (A),\]
but are not contained in  $\widetilde{E}^{(\alpha\beta)}_{ij} (A)$ (or $\widehat{E}^{(\alpha\beta)}_{ij} (A)$).
On the other hand, it should be pointed out here that
\begin{equation}\label{neq3.1} E^{(\alpha \beta)} (A) = G^{(\alpha \beta)} (A)\end{equation} holds in some cases.
Also consider the matrix $A$ above, and  take $\alpha=\{1,4\}$ and $\beta=\{2,3\}$, in which case (\ref{neq3.1}) holds.

%\begin{figure}[ft]
%\centering
%\includegraphics[width=3in]{fig3.jpg}
%\caption{all eigenvalues of $A$ and $E^{(\alpha \beta)} (A)$ }
%\end{figure}
\end{rmk}

It is well-known that an eigenvalue inclusion region in the complex can give a sufficient condition of the non-singularity of a matrix \cite{Cv1}. Hence, by Theorem \ref{Mian_thm} we can get easily  the following result.

\begin{corollary}
Let $A=[a_{ij}]\in \mathbb{C}^{n\times n}$. If for each $i\in\alpha$, $j\in\beta$, the following holds:

(I) \[|a_{ii}|> r_i^\alpha(A),\]

(II)\[|a_{jj}|> r_j^{\beta}(A),\]

(III) \[(|a_{ii}|-  r_i^\alpha(A))(|a_{jj}|-r_j^{\beta}(A))>r_i^{\beta}(A)r_j^\alpha(A),\]
or \[(|a_{ii}|+r_i^j(A))(|a_{jj}|+r_j^\beta(A))<|a_{ij}|(2|a_{ji}|-r_j^\alpha(A)),\]

(IV) \[(|a_{ii}|-  r_i^\alpha(A))(|a_{jj}|-r_j^{\beta}(A))>r_i^{\beta}(A)r_j^\alpha(A),\]
 or
\[(|a_{jj}|+r_j^i(A))(|a_{ii}|+r_i^\alpha(A))<|a_{ji}|(2|a_{ij}|-r_i^\beta(A)),\]
\end{corollary}
then $A$ is nonsingular.
%%%%%%%%%%%%%%%%%%%%
\section{Conclusions}
In this paper, a new eigenvalue inclusion region $E^{(\alpha \beta)} (A)$  is given by excluding some regions  $\widetilde{E}^{(\alpha\beta)}_{ij} (A)$ and  $\widehat{E}^{(\alpha\beta)}_{ij} (A)$ from the $\alpha\beta$-type eigenvalue inclusion region $G^{(\alpha \beta)} (A).$ As shown in Remark \ref{rmk1},  $E^{(\alpha \beta)} (A)=G^{(\alpha \beta)} (A) $ holds in some cases, hence besides $\widetilde{E}^{(\alpha\beta)}_{ij} (A)$ and  $\widehat{E}^{(\alpha\beta)}_{ij} (A)$, it is very interesting to find other regions which do not contain any eigenvalues of a matrix to exclude them. Furthermore, for the well-known eigenvalue inclusion regions, such as, regions in \cite{Bra,Bru,Cv0,Cv1,Cv2,Ge,Ko,Li,Me1}, we can try to find some regions like $\widetilde{E}^{(\alpha\beta)}_{ij} (A)$ and  $\widehat{E}^{(\alpha\beta)}_{ij} (A)$, and exclude them from the corresponding existing eigenvalue inclusion regions to give new regions which capture all eigenvalues of a matrix more precisely.

%%%%%%%%%%%%%%%%%%%%%%%%%%%%%%%%%%%%%%%%%%%%%%%%%%%%%%%%%%%%%%%%%%%%%%
\section*{Acknowledgements}
This paper  is dedicated to Professor *** on the occasion of his 60th birthday. This work is supported in part by National Natural Science Foundations of
China (11601473 and 11361074), the National Natural Science Foundation of Zhejiang Province (LY14A010007, LQ14G010002), Ningbo Natural
Science Foundation (2015A610173), and CAS "Light of West China" Program.

\end{document}